\documentclass[12pt]{amsart}

\usepackage{geometry}                
\usepackage{graphicx}
\usepackage{amsthm,amssymb,amsmath}
\usepackage{subfigure}

\begin{document}
\title{The Genus of a Random Bipartite Graph}

\author{Yifan Jing}
\address{%
Department of Mathematics\\
Simon Fraser University\\
Burnaby, BC, Canada}
\email{yifanjing17@gmail.com}

\author{Bojan Mohar}
\address{%
Department of Mathematics\\
Simon Fraser University\\
Burnaby, BC, Canada}
\email{mohar@sfu.ca}

\thanks{B.M.~was supported in part by the NSERC Discovery Grant R611450 (Canada), by the Canada Research Chairs program, and by the Research Project J1-8130 of ARRS (Slovenia).}
\thanks{On leave from IMFM \& FMF, Department of Mathematics, University of Ljubljana.}%

\subjclass[2010]{Primary 05C10; Secondary 57M15}

\date{} 

\maketitle

\newtheorem{theorem}{Theorem}[section]
\newtheorem{lemma}[theorem]{Lemma}
\newtheorem{definition}[theorem]{Definition}
\newtheorem*{thm:associativity}{Theorem \ref{thm:associativity}}
\newtheorem*{thm:associativity2}{Theorem \ref{thm:associativity2}}
\newtheorem*{thm:associativity3}{Theorem \ref{thm:associativity3}}

\begin{abstract}
Archdeacon and Grable (1995) proved that the genus of the random graph $G\in\mathcal{G}_{n,p}$ is almost surely close to $pn^2/12$ if $p=p(n)\geq3(\ln n)^2n^{-1/2}$. In this paper we prove an analogous result for random bipartite graphs in $\mathcal{G}_{n_1,n_2,p}$.
If $n_1\ge n_2 \gg 1$, phase transitions occur for every positive integer $i$ when $p=\Theta((n_1n_2)^{-\frac{i}{2i+1}})$.
A different behaviour is exhibited when one of the bipartite parts has constant size, i.e. $n_1\gg1$ and $n_2$ is a constant. In that case, phase transitions occur when $p=\Theta(n_1^{-1/2})$ and when $p=\Theta(n_1^{-1/3})$.
\end{abstract}

\section{Introduction}



For a simple graph $G$, let $g(G)$ be the {\em genus} of $G$, that is, the minimum $h$ such that $G$ embeds into the orientable surface $\mathbb{S}_h$ of genus $h$, and let $\widetilde{g}(G)$ be the {\em non-orientable genus} of $G$ which is the minimum $c$ such that $G$ embeds into the non-orientable surface $\mathbb{N}_c$ with crosscap number $c$. The surface here is a compact two-dimensional manifold without boundary. We say $G$ is {\em 2-cell embedded} in a surface $S$ if each face of $G$ is homeomorphic to an open disk, and a {\em$k$-gon embedding} of $G$ is a 2-cell embedding in which every face is bounded by a cycle of length $k$.

Given a graph $G$, determining the genus of $G$ is one of the fundamental problems in topological graph theory. Youngs \cite{Y} showed that the problem of determining the genus of a connected graph $G$ is the same as determining a $2$-cell embedding of $G$ with minimum genus. The same holds for the non-orientable genus \cite{non}. It was proved by Thomassen \cite{NPC} that the genus problem is NP-complete. For further background on topological graph theory, we refer to \cite{top}.

The random graph $\mathcal{G}_{n,p}$ is a probability space whose objects are all (labelled) graphs defined on a vertex set $V$ of cardinality $n$, and each possible edge occurs with probability $p$ independently, i.e., a graph $G=(V,E)\in\mathcal{G}_{n,p}$ has probability $p^{|E|}(1-p)^{\binom{n}{2}-|E|}$. Similarly, one can define {\em random bipartite graphs} $\mathcal{G}_{n_1,n_2,p}$ as the probability space of all bipartite graphs with (labelled) bipartition $X\sqcup Y$, $|X|=n_1$, $|Y|=n_2$, where each edge $xy$ ($x\in X$, $y\in Y$) appears with probability $p$. In this paper we will always assume $n_1\geq n_2$ for the convenience. There are thousands of papers studying properties of random graphs; for more background about this fascinating area, see \cite{prob, ran}.

Stahl \cite{Stahl} was the first to consider the genus (in fact, the average genus) of random graphs. Almost concurrently, Archdeacon and Grable \cite{dense} studied the genus of random graphs in $\mathcal{G}_{n,p}$. They obtained the following result when $p=p(n)$ is not too small.

\begin{theorem}[Archdeacon and Grable \cite{dense}]\label{thm:1.1}
Let $\varepsilon>0$ and let $0<p<1$ with $p^2(1-p^2)\geq 8(\ln n)^4/n$. Then almost every graph $G$ in $\mathcal{G}_{n,p}$ satisfies
\begin{equation*}
(1-\varepsilon)\frac{pn^2}{12}\leq g(G)\leq(1+\varepsilon)\frac{pn^2}{12}
\end{equation*}
and
\begin{equation*}
(1-\varepsilon)\frac{pn^2}{6}\leq \widetilde{g}(G)\leq(1+\varepsilon)\frac{pn^2}{6}.
\end{equation*}
\end{theorem}

They also conjectured that almost every graph in $\mathcal{G}_{n,p}$ has an $\varepsilon$-near $k$-gon embedding (in which all but an $\varepsilon$-fraction of edges lie on the boundary of two $k$-gonal faces) on some orientable surface and on some non-orientable surface. R{\"o}dl and Thomas \cite{genus} resolved their conjecture and extended Theorem \ref{thm:1.1} to an even broader range of edge-probabilities.

\begin{theorem}[R{\"o}dl and Thomas \cite{genus}]\label{thm:1.2}
Let $\varepsilon>0$, let $i\geq1$ be an integer and assume that $n^{-\frac{i}{i+1}}\ll p\ll n^{-\frac{i-1}{i}}$. Then $G\in\mathcal{G}_{n,p}$ almost surely satisfies
\begin{equation*}
(1-\varepsilon)\frac{i}{4(i+2)}pn^2\leq g(G)\leq(1+\varepsilon)\frac{i}{4(i+2)}pn^2
\end{equation*}
and
\begin{equation*}
(1-\varepsilon)\frac{i}{2(i+2)}pn^2\leq \widetilde{g}(G)\leq(1+\varepsilon)\frac{i}{2(i+2)}pn^2.
\end{equation*}
\end{theorem}

In this paper, we will study the genus of random bipartite graphs, which plays an important role in approximating the genus of dense graphs \cite{JM}. The main results of this paper show that a result similar to Theorems \ref{thm:1.1} and \ref{thm:1.2} is also true for random bipartite graphs.

\begin{theorem}\label{thm:associativity}\label{thm:1.3}
Let $\varepsilon>0$ and $G\in\mathcal{G}_{n_1,n_2,p}$ be a random bipartite graph and suppose that $i\geq2$ is an integer. If $p$ satisfies $(n_1n_2)^{-\frac{i}{2i+1}}\ll p\ll (n_1n_2)^{-\frac{i-1}{2i-1}}$, $n_1/n_2<c$ and $n_2/n_1<c$ where $c$ is a positive real number, then we have a.a.s.
\begin{equation*}
(1-\varepsilon)\frac{i}{2i+2}pn_1n_2\leq g(G)\leq(1+\varepsilon)\frac{i}{2i+2}pn_1n_2
\end{equation*}
and
\begin{equation*}
(1-\varepsilon)\frac{i}{i+1}pn_1n_2\leq \widetilde{g}(G)\leq(1+\varepsilon)\frac{i}{i+1}pn_1n_2.
\end{equation*}
\end{theorem}

In Theorem \ref{thm:1.3} we request that $i\geq2$. For $i=1$ we have a stronger result that we discuss next. In this case $p$ is relatively large and we will show that $G\in\mathcal{G}_{n_1,n_2,p}$ will almost surely have an $\varepsilon$-near $4$-gon embedding.

\begin{theorem}\label{thm:associativity2}\label{thm:1.4}
Let $\varepsilon>0$ and $G\in\mathcal{G}_{n_1,n_2,p}$ be a random bipartite graph. If $n_1\geq n_2\gg1$ and $p\gg n_2^{-\frac{2}{3}}$, then we have a.a.s.
\begin{equation*}
(1-\varepsilon)\frac{pn_1n_2}{4}\leq g(G)\leq(1+\varepsilon)\frac{pn_1n_2}{4}
\end{equation*}
and
\begin{equation*}
(1-\varepsilon)\frac{pn_1n_2}{2}\leq \widetilde{g}(G)\leq(1+\varepsilon)\frac{pn_1n_2}{2}.
\end{equation*}
\end{theorem}

The above results exhibit phase transitions for every positive integer $i$, when $p=\Theta((n_1n_2)^{-\frac{i}{2i+1}})$. The genus in these critical ranges can be estimated within a constant factor as follows. Let $n=\sqrt{n_1n_2}$ and $\varepsilon>0$. It is easy to see that the genus of a graph $G$ satisfies the {\em edge-Lipschitz condition}, i.e., if $G$ and $G^\prime$ differ in only one edge, then $|g(G)-g(G^\prime)|\leq1$. For random variables satisfying the Lipschitz condition, one can analyze what happen when a phase transition occurs. The results in \cite[Chapter 7]{prob} show that the following holds. When $n_2=\Theta(n_1)$ and $p=cn^{-\frac{2i}{2i+1}}$ for $i\geq2$, there exists a number $f(c,n,p)$ with $\frac{i}{2i+2}\leq f(c,n,p)\leq\frac{i+1}{2i+4}$ such that $G\in\mathcal{G}_{n_1,n_2,p}$ will have
$$(1-\varepsilon)f(c,n,p)\, pn^2 \leq g(G) \leq (1+\varepsilon)f(c,n,p)\, pn^2, \quad \textrm{a.a.s.}$$

When a random bipartite graph $G\in\mathcal{G}_{n_1,n_2,p}$ satisfies $n_1\gg1$ and $n_2$ is a constant, the genus of $G$ has different behaviour.

\begin{theorem}\label{thm:associativity3}\label{thm:1.5}
Let $\varepsilon>0$ and $G\in\mathcal{G}_{n_1,n_2,p}$ where $n_1\gg 1$ and $n_2\geq3$ is a constant.
\begin{itemize}
\item[\rm{(a)}]
If $p\gg n_1^{-\frac{1}{3}}$ we have a.a.s. (as $n_1\to\infty$)
\begin{equation*}
(1-\varepsilon)\frac{n_1n_2p}{4}\,\Psi(p,n_2)\leq g(G)\leq(1+\varepsilon)\frac{n_1n_2p}{4}\,\Psi(p,n_2)
\end{equation*}
and
\begin{equation*}
(1-\varepsilon)\frac{n_1n_2p}{2}\,\Psi(p,n_2)\leq \widetilde{g}(G)\leq(1+\varepsilon)\frac{n_1n_2p}{2}\,\Psi(p,n_2),\\[1mm]
\end{equation*}
where $\Psi(p,n_2)=\sum_{i=2}^{n_2-1}\frac{i-1}{i+1}\binom{n_2-1}{i}(-p)^i$.
\item[\rm{(b)}]
If $n_1^{-\frac{1}{2}}\ll p\ll n_1^{-\frac{1}{3}}$, then a.a.s.
\begin{equation*}
g(G)=\biggl\lceil\frac{(n_2-3)(n_2-4)}{12}\biggr\rceil\quad and\quad \widetilde{g}(G)=\biggl\lceil\frac{(n_2-3)(n_2-4)}{6}\biggr\rceil
\end{equation*}
with a single exception that $\widetilde{g}(G)=3$ when $n_2=7$.
\item[\rm{(c)}]
If $p\ll n_1^{-\frac{1}{2}}$, then a.a.s. $g(G)=0$.
\end{itemize}
\end{theorem}

This result shows two phase transitions. When $p\ll n_1^{-1/2}$, a random bipartite graph is almost surely planar; after this first threshold, we obtain a subdivision of the complete graph on $n_2$ vertices (with additional vertices of degrees 0 or 1), and when $p\gg n_1^{-1/3}$, $G$ has an $\varepsilon$-near $4$-gon embedding a.a.s.

The paper is organized as follows. In the next section, we give basic definitions and properties in topological graph theory and discuss random graphs. Also, our main tools used in the proofs are presented. In Section $3$, we prove Theorems \ref{thm:1.3} and \ref{thm:1.4}. Section $4$ resolves the cases when one of the bipartition parts has constant size and contains the proof of Theorem~\ref{thm:1.5}.

\section{Preliminaries}%

We will use standard definitions and notation for graphs and probabilistic methods as given in \cite{Diestel} and \cite{prob,ran}. We use the following notation: $A(n)\sim B(n)$ means $\lim_{n\to\infty}A(n)/B(n)=1$, and $A(n)\ll B(n)$ means $\lim_{n\to\infty}A(n)/B(n)=0$. By $X\sqcup Y$ we denote the disjoint union of $X$ and $Y$, and we set $X\oplus Y=(X\times Y)\sqcup (Y\times X)$. We say an event $A(n)$ happens {\em asymptotically almost surely} (abbreviated {\em a.a.s.}) if $\mathbb{P}(A(n))\to1$ as $n\to\infty$.

We consistently use $G$ to denote a simple undirected graph, $D$ is always a digraph and $\mathcal{H}$ is a hypergraph. A vertex partition $\mathcal{P}=\{V_i\}_{i=1}^k$ is {\em equitable} if $V_i\cap V_j=\varnothing$ for every $1\leq i<j\leq k$ and the parts have size as equal as possible, i.e. $||V_i|-|V_j||\leq 1$ for all $i,j$. A {\em trail} in a graph $G$ (or a digraph $D$) is a (directed) walk that has no repeated edges. A {\em closed trail} is a trail that starts and ends at the same vertex. If $D$ is a digraph, then $D^{-1}$ is the digraph obtained from $D$ by replacing each arc $\overrightarrow{xy}$ with the reverse arc $\overrightarrow{yx}$.

Let $G$ be a simple graph. A {\em corresponding digraph} of $G$ is a random simple digraph $\mathcal{D}$ obtained from $G$ by randomly orienting each edge. Specifically, each digraph $D\in\mathcal{D}$ has $V(D)=V(G)$ and if $uv\in E(G)$ then either $\overrightarrow{uv}$ or $\overrightarrow{vu}$ is an edge of $D$, each has probability $1/2$ and the two events are exclusive. Also, the choices made for different edges are independent from each other. The corresponding digraph $\mathcal{D}$ of a random graph $\mathcal{G}$ is a family of digraphs defined on the same vertex set of graphs in $\mathcal{G}$, and when two vertices $u,v$ produce an edge with probability $p$ in $\mathcal{G}$, then $\overrightarrow{uv}$ occurs with probability $\frac{p}{2}$ and $\overrightarrow{vu}$ occurs with probability $\frac{p}{2}$ in $\mathcal{D}$, and those two events are exclusive.

Now we focus on the $2$-cell embeddings of a graph $G$. We say $\Pi=\{\pi_v\,|\,v\in V(G)\}$ is a {\em rotation system} if for each vertex $v$, $\pi_v$ is a cyclic permutation of the edges incident with $v$. The {\em Heffter-Edmonds-Ringel rotation principle} \cite[Theorem 3.2.4]{top} shows that every $2$-cell embedding of a connected graph $G$ in an orientable surface is uniquely determined (up to homeomorphisms of the surface) by its rotation system. Let $g(G)$ be the orientable genus of $G$ and let $\widetilde{g}(G)$ be the non-orientable genus of $G$. For $2$-cell embeddings we have the famous Euler's Formula.
\begin{theorem}[Euler's Formula]
Let $G$ be a connected graph which is 2-cell embedded in a surface $S$. If $G$ has $n$ vertices, $e$ edges and $f$ faces in $S$, then
\begin{equation}
\chi(S)=n-e+f.
\end{equation}
\end{theorem}
Here $\chi(S)$ is the {\em Euler characteristic} of the surface $S$, where $\chi(S)=2-2h$ when $S=\mathbb{S}_h$ and $\chi(S)=2-c$ when $S=\mathbb{N}_c$.

Given a digraph $D$, a {\em blossom} of length $l$ with {\em center} $v$ and {\em tips} $\{v_1,v_2,\dots,v_l\}$ is a set $\mathcal{C}$ of $l$ directed cycles $\{ C_1,C_2,\dots,C_l\}$, where $\overrightarrow{v_{i}v},\overrightarrow{vv_{i+1}}\in C_i$, for $i=1,2,\dots,l$, with $v_{l+1}=v_1$. A $k$-blossom is a blossom, all of whose elements are directed $k$-cycles. A blossom of length $l$ is {\em simple} if either $l\geq3$ or $l=2$ and $C_1\neq C_2^{-1}$.
\begin{figure}[h]
\centering
\includegraphics[width=2.2in]{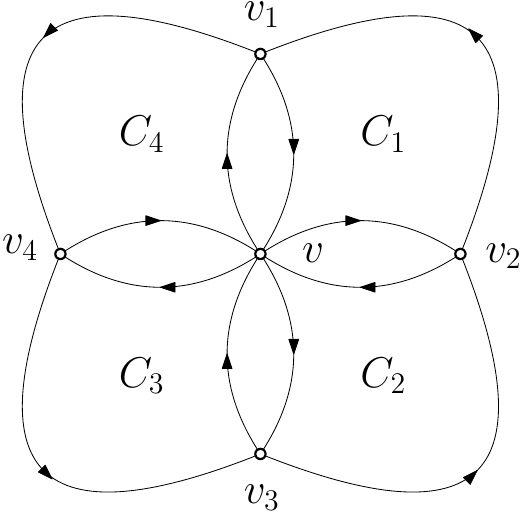}
\caption{A 3-blossom of length $4$ with center $v$ and tips $v_1,v_2,v_3,v_4$.}
\end{figure}

Let $\mathcal{C}$ be a family of arc-disjoint closed trails in $D\cup D^{-1}$. We say that $\mathcal{C}$ is {\em blossom-free} if no subset of $\mathcal{C}$ forms a blossom centered at some vertex. The following lemma is a slight strengthening of \cite[Lemma 2.1]{genus}; the proof is elementary and we omit details.
\begin{lemma}\label{lem:blossom}
Let $G$ be a graph and let $D$ be a corresponding digraph. Suppose that $\mathcal{C}_1$ and $\mathcal{C}_2$ are sets of arc-disjoint closed trails in $D$ and $D^{-1}$ (respectively) such that their union $\mathcal{C}_1\cup\mathcal{C}_2$ is blossom-free in $D\cup D^{-1}$. Then there exists a rotation system $\Pi$ of $G$ such that every closed trail in $\mathcal{C}_1\cup\mathcal{C}_2$ is a face of $\Pi$.
\end{lemma}
For every $\varepsilon>0$, an $\varepsilon$-{\em near} $k$-{\em gon embedding} $\Pi$ is a rotation system of $G$ such that $kf_k(\Pi)\geq2(1-\varepsilon)|E(G)|$, where $f_k(\Pi)$ is the number of faces of length $k$ of $\Pi$.

The following result from \cite{FR} (see also \cite{PS,genus} where its current formulation appears) will be our main tool for constructing near-optimal embeddings of random graphs.
\begin{theorem}\label{thm:hypergraphmatching}
Let $\varepsilon>0$ be a real number and $d\geq2$ be an integer. Then there exist a positive real number $\delta$ and an integer $N_0$ such that for every $N\geq N_0$ the following holds. If $\Delta$ is a real number and if $\mathcal{H}$ is a $d$-uniform hypergraph with $|V(\mathcal{H})|=N$ such that
\begin{enumerate}
\item $|\{x\in V(\mathcal{H})\ |\ (1-\delta)\Delta\leq \deg(x)\leq(1+\delta)\Delta\}|\geq(1-\delta)N$,
\item for every $x,y\in V(\mathcal{H})$ with $x\neq y$, $|\{e\in E(H)\mid x,y\in e\}|<\delta\Delta$,
\item at most $\delta N\Delta$ hyperedges of $\mathcal{H}$ contain a vertex $v\in V(\mathcal{H})$ with $\deg(v)>(1+\delta)\Delta$,
\end{enumerate}
then $\mathcal{H}$ has a matching of size at least $(1-\varepsilon)N/d$. Moreover, for every matching $M$ in $\mathcal{H}$, there exists a matching $M^\prime$ in $\mathcal{H}$ with $M\cap M^\prime=\varnothing$, and with $|M^\prime|\geq(1-\varepsilon)N/d$.
\end{theorem}

Similarly as for undirected graphs (see \cite[Lemma 5.4.2]{top}), we have the following property on digraphs. We say that two paths from $x$ to $y$ are \emph{internally disjoint} if they have no other vertices in common except $x$ and $y$. The same definition applies in the case of closed paths (when $x=y$).

\begin{lemma}\label{lem:path}
For any positive integers $f$ and $k$, there exists an integer $K=K(f,k)$ such that the following holds for every
digraph $D$.  If $D$ contains vertices $x,y\in V(D)$ (possibly $x=y$) and there are $K$ pairwise different $(x,y)$-trails of length $f$ in $D$, then there exist vertices $u,v\in V(D)$ (possibly $u=v$) and $k$ internally disjoint $(u,v)$-paths (closed paths if $u=v$), all of the same length $l$, where $1\leq l\leq f$.
\end{lemma}

\begin{proof}
The proof is by induction on $f$ for the values $K(f,k)$ defined recursively: $K(1,k)=k$ and $K(f,k) = 2kf^3\,K(f-1,k)^2$. When $k=0$ there is nothing to prove, and if $f=1$, we have $k$ parallel edges (loops if $x=y$) joining $x$ and $y$, and these form the desired internally disjoint paths (or closed paths).

Suppose now that $k>0$ and $f\ge2$ and that we have a set $\mathcal T$ of $K(f,k)$ trails of length $f$ from $x$ to $y$.
If one of the edges $az$ is the initial arc of $K(f-1,k)$ of the trails in $\mathcal T$, then we apply the induction to the corresponding family of $(z,y)$-trails of length $f-1$. Otherwise, there is a subset $\mathcal T'$ of $\mathcal T$ containing
$K(f,k)/K(f-1,k) = 2kf^3 K(f-1,k)$ trails from $x$ to $y$, whose initial edges are all distinct. Each such trail $T$ either contains a cycle (closed path from $x$ and back to $x$) or a path from $x$ to $y$ which starts with the same edge as $T$. We may assume that we have a path from $x$ to $y$ in at least half of the cases (otherwise we similarly treat closed paths at $x$). So we have a set $\mathcal P$ of $kf^3K(f-1,k)$ paths from $x$ to $y$, all starting with different edges.

Suppose that $\mathcal P$ contains a path $P$ such that $f^3K(f-1,k)$ other paths intersect it internally. At least $f^2K(f-1,k)$ of those paths intersect $P$ at the same vertex $z$ and at least $K(f-1,k)$ of those have the same length $l$ from $x$ to $z$. We complete the proof by applying induction to this set of paths of length $l<f$.
If the above case does not apply, each path in $\mathcal P$ intersects internally fewer than $f^3K(f-1,k)$ other paths. Thus, $\mathcal P$ contains at least $fk$ paths that are pairwise internally disjoint, of which at least $k$ must have the same length.
\end{proof}



\section{Genus of random bipartite graphs}%

In this section we treat random bipartite graphs in $\mathcal{G}_{n_1,n_2,p}$. Let us first consider the case when $n_1$ and $n_2$ have about the same magnitude.

\begin{lemma}\label{lem:biembed}
Let $\varepsilon>0$ and $G\in\mathcal{G}_{n_1,n_2,p}$ be a random bipartite graph on vertex set $X\sqcup Y$ with $|X|=n_1\geq n_2=|Y|$. If there exist a positive real number $c$ and a positive integer $i$ such that $n_1/n_2<c$, and $p\gg (n_1n_2)^{-\frac{i}{2i+1}}$, then a.a.s. $G$ has an $\varepsilon$-near $(2i+2)$-gon embedding.
\end{lemma}

\begin{proof}
Choose $0<\varepsilon_1<\frac{1}{2}$, $\varepsilon_0=\frac{3i+4}{1-\varepsilon_1}\varepsilon_1$, such that $\varepsilon_0<1/2$ and $\varepsilon\geq\frac{5\varepsilon_0}{1+\varepsilon_0}$. Let $n=\sqrt{n_1n_2}$. Then $p\gg n^{-\frac{2i}{2i+1}}$. Let us first assume that $p\ll n^{-\frac{2i-\varepsilon_1}{2i+1-\varepsilon_1}}$. Let $D\in\mathcal{D}$ be the corresponding digraph of $\mathcal{G}_{n_1,n_2,p}$. Consider the following hypergraph $\mathcal{H}$, where $V(\mathcal{H})$ is the edge set of $D$ and $E(\mathcal{H})$ is the set of closed trails of $D$ of length $2i+2$. Let $d=2i+2,\delta=\frac{\varepsilon_1}{1-\varepsilon_1}$ and $\Delta=n_1^{i}n_2^{i}(\frac{p}{2})^{2i+1}$. We claim that our hypergraph $\mathcal{H}$ satisfies all three conditions in Theorem \ref{thm:hypergraphmatching}, a.a.s.

To prove that condition (1) holds, let $N=|V(\mathcal{H})|$. We have
\begin{equation}
\begin{split}
&\mathbb{E}(N)=n_1n_2p,\\
&\mathbb{E}(N^2)=n_1n_2p(n_1-1)(n_2-1)p+O(n_1^2n_2p^2+n_2^2n_1p^2).
\end{split}
\end{equation}
By Chebyshev's inequality,
\begin{equation}
\mathbb{P}(|N-\mathbb{E}(N)|\geq\varepsilon_1n_1n_2p)\leq\frac{\mathbb{E}(N^2)-\mathbb{E}^2(N)}{\varepsilon_1^2\mathbb{E}^2(N)}=O\biggr(\frac{n_1+n_2}{n_1n_2}\biggr)=o(1).
\end{equation}
Therefore, we have a.a.s.
\begin{equation}\label{N}
(1-\varepsilon_1)n_1n_2p<N<(1+\varepsilon_1)n_1n_2p.
\end{equation}
For each pair of vertices $(a,b)\in X\oplus Y$, let $\rho(b,a)$ be the number of directed trails in $D$ from $b$ to $a$ of length $2i+1$, and let $U$ be the number of edges $\overrightarrow{uv}$ of $D$ such that the number of directed trails from $v$ to $u$ of length $2i+1$ is at most $(1-\delta)\Delta$ or at least $(1+\delta)\Delta$. We have
\begin{equation}
\begin{split}
\mathbb{E}(\rho(b,a))\geq&\ \binom{n_1-1}{i}\binom{n_2-1}{i}(i!)^2\bigl(\frac{p}{2}\bigl)^{2i+1},\text{ and}\\
\mathbb{E}(\rho(b,a))\leq&\ n_1^in_2^i(i!)^2\bigl(\frac{p}{2}\bigl)^{2i+1}.
\end{split}
\end{equation}
The conclusion is that $\mathbb{E}(\rho(b,a))=\Theta(n_1^in_2^ip^{2i+1})$, where the lower bound is obtained by counting $(a,b)$-paths and the upper bound is obtained by counting $(a,b)$-walks. For $\rho^2(b,a)$ we have
\begin{equation}
\begin{split}
\mathbb{E}(\rho^2(b,a))=\mathbb{E}^2(\rho(b,a))+O(n_1^{2i}n_2^{2i-1}p^{4i+1}+n_1^{2i-1}n_2^{2i}p^{4i+1}).
\end{split}
\end{equation}
Using Chebyshev's inequality, since $|\Delta-\mathbb{E}(\rho(b,a))|=o(\mathbb{E}(\rho(b,a)))$, for sufficiently large $n$,
\begin{equation}\label{good}
\begin{split}
\mathbb{P}\bigl(|\rho(b,a)-\Delta|\geq\delta\Delta\bigr)&\leq\mathbb{P}\Bigl(|\rho(b,a)-\mathbb{E}(\rho(b,a))|\geq\frac{\varepsilon_1}{2}\mathbb{E}(\rho(b,a))\Bigr)\\
&\leq\frac{\mathbb{E}(\rho^2(b,a))-\mathbb{E}^2(\rho(b,a))}{(\frac{\varepsilon_1}{2})^2\mathbb{E}^2(\rho(b,a))}=O\biggl(\frac{n_1+n_2}{n_1n_2p}\biggl)=o(1).
\end{split}
\end{equation}
Also for $U$ we have
\begin{equation}
\mathbb{E}(U)=pn_1n_2\mathbb{P}(|\rho(b,a)-\Delta|\geq\delta\Delta)\leq O(n_1+n_2).
\end{equation}
Hence by Markov's inequality,
\begin{equation}\label{condition1}
\mathbb{P}\Bigl(U\geq\varepsilon_1\frac{p}{2}n_1n_2\Bigr)\leq\frac{\mathbb{E}(U)}{\varepsilon_1\frac{p}{2}n_1n_2}=O\biggl(\frac{n_1+n_2}{n_1n_2p}\biggl)=o(1).
\end{equation}
This means, together with (\ref{N}) and (\ref{good}), a.a.s. at least $N-\varepsilon_1\frac{p}{2}n_1n_2>(1-\delta)N$ vertices of $V(\mathcal{H})$ satisfy $(1-\delta)\Delta\leq\mathrm{deg}(x)\leq(1+\delta)\Delta$, so condition (1) holds for $\mathcal{H}$.

To verify (2), let $e,f$ be two edges of $D$ that together belong to at least $\delta\Delta$ hyperedges in $\mathcal{H}$. This means that they are together in many closed trails of length $2i+2$. We may assume that the number of these closed trails is more than $L = (2i\,K(2i,8i+2))^2$, where $K(.,.)$ is the value from Lemma \ref{lem:path}.
Let us consider the subtrails of these closed trails from the endvertex of $e$ to the initial vertex of $f$ and the subtrails from the endvertex of $f$ to the initial vertex of $e$. One family of the subtrails, say the trails from the endvertex of $e$ to the initial vertex of $f$, contains at least $\sqrt{L} = 2iK(2i,2i+8)$ distinct trails, all of length at most $2i$. A subset of $K(2i,2i+8)$ of these trails has the same length $l\le 2i$.
By Lemma \ref{lem:path}, there exist vertices $u$ and $v$ (possibly $u=v$), and at least $8i+2$ internally disjoint directed (closed) paths from $u$ to $v$ of length $l$, where $2\leq l\leq 2i$.

Let $B$ be the number of vertex pairs $(u,v)\in V(D)^2$ such that there exist $8i+2$ internally disjoint directed paths from $u$ to $v$ of length $l$. Note that $p\ll n^{-\frac{2i-\varepsilon_1}{2i+1-\varepsilon_1}}<n^{-\frac{4i-1}{4i+1}}$ since $\varepsilon_1<1/2$. We have
\begin{equation}
\begin{split}
\mathbb{E}(B)&=O\Bigl(n_1^{(8i+2)\frac{l-1}{2}+1}n_2^{(8i+2)\frac{l-1}{2}+1}p^{(8i+2)l}\Bigr)\leq o(n^{4l-8i})=o(1), \text{ when } l\equiv1\ (\mathrm{mod}\, 2);\\
\mathbb{E}(B)&=O\Bigl(n_1^{(8i+2)\frac{l}{2}}n_2^{(8i+2)\frac{l-2}{2}+2}p^{(8i+2)l}\Bigr)\leq o(n_1^{2l}n_2^{2l-8i})=o(1), \text{ when } l\equiv0\ (\mathrm{mod}\, 2).
\end{split}
\end{equation}
By Markov's inequality, $\mathbb{P}(B\geq1)\leq o(1)$. This implies that for every $e,f\in A(D)$, at most $L$ closed trails in $D$ contain both $e$ and $f$, a.a.s. Therefore in our hypergraph $\mathcal{H}$, condition (2) holds for $\mathcal{H}$ when $n$ is large enough.

Finally, let us consider condition (3) of Theorem \ref{thm:hypergraphmatching}. Let $F$ be the number of closed trails of length $2i+2$ in $D$ which contain at least one directed edge $\overrightarrow{uv}\in P^\delta$, where $P^\delta$ is the set of pairs of vertices $(u,v)\in X\oplus Y$ such that the number of directed trails from $v$ to $u$ of length $2i+1$ is at least $(1+\delta)\Delta$. Each trail $R = x_1y_1x_2y_2\cdots y_{i+1}x_1$ contributing to $F$ is determined by two sequences of vertices $x_1,x_2,\dots,x_{i+1}\in X$ and $y_1,y_2,\dots,y_{i+1}\in Y$. Each such closed trail $R$ has the same probability that it forms a trail contributing to $F$. There are $2i+2$ candidates for an edge of $R$ being in $P^\delta$. This implies that
\begin{equation}\label{conditionalexp}
\mathbb{E}(F)\leq n_1^{i+1}n_2^{i+1}(2i+2)\mathbb{P}(R\subseteq A(D))\,\mathbb{P}(\overrightarrow{x_1y_1} \in P^\delta\mid R\subseteq A(D)).
\end{equation}

For $j=1,\dots,2i-1$, let $\alpha_j$ be the number of trails of length $2i+1$ from $y_1$ to $x_1$ that contain precisely $j$ edges in $R$. Then $\alpha = \sum_{j=1}^{2i-1} \alpha_j$ is the number of trails of length $2i+1$ from $y_1$ to $x_1$ different from $R$ which contain at least one edge in $R$. Since $n_2=\Theta(n_1)$, we have
\begin{equation}\label{alpha}
\begin{split}
\mathbb{E}(\alpha \mid R\subseteq A(D))
  &= \sum_{j=1}^{2i-1} \mathbb{E}(\alpha_j \mid R\subseteq A(D))\\[1mm]
  &\leq \sum_{j=1}^{2i-1} \binom{2i+1}{j}^2\,j!\, n_1^{2i-j} \bigr(\frac{p}{2}\bigr)^{2i+1-j}\\[1mm]
  &\leq O(n^{2i-1}p^{2i+1}).
\end{split}
\end{equation}
Now, by Markov's inequality,
\begin{equation}
\mathbb{P}(\alpha\geq1 \mid R\subseteq A(D))\leq O(n^{2i-1}p^{2i+1})\ll O(n^{-\frac{4i}{4i+1}})=o(1).
\end{equation}

In the next argument we will use the following events: $Q^\delta$ is the event that the number of trails of length $2i+1$ from $y_1$ to $x_1$ that are different from $R$ is at least $(1+\delta)\Delta-1$; $R_E$ is the event that all edges in $R$ appear in $D$, possibly with different orientations. There are $2^{2i+2}$ different orientations $\omega_1,\dots, \omega_{2^{2i+2}}$ of these edges. We denote by $R_E^j$ the event that these edges are present and have orientation $\omega_j$. Clearly, different events $R_E^j$ are mutually exclusive and $R_E$ is the union of all these events. Note that the following holds:
\begin{equation*}
\begin{split}
   \mathbb{P}(\overrightarrow{x_1y_1} \in P^\delta, \alpha=0 \mid R\subseteq A(D))
   &= \mathbb{P}(\overrightarrow{x_1y_1} \in Q^\delta, \alpha=0 \mid R\subseteq A(D))\\
   &\le \sum_{j=1}^{2^{2i+2}} \mathbb{P}(\overrightarrow{x_1y_1} \in Q^\delta, \alpha=0 \mid R_E^j) \\[1mm]
   &= 2^{2i+1}\, \mathbb{P}(\overrightarrow{x_1y_1} \in Q^\delta, \alpha=0 \mid R_E)\\[0.5mm]
   &\le 2^{2i+1}\, \mathbb{P}(\overrightarrow{x_1y_1} \in Q^\delta, \alpha=0)\\
   &\le 2^{2i+1}\, \mathbb{P}(\overrightarrow{x_1y_1} \in Q^\delta).
\end{split}
\end{equation*}
We used the fact that $\alpha=0$ is less likely to happen under the condition that $R_E$ holds and that $\overrightarrow{x_1y_1} \in Q^\delta$ is independent of $R_E$ when $\alpha=0$.

Combining the above inequalities with~(\ref{good}), we get
\begin{equation}\label{13}
\begin{split}
\mathbb{P}(\overrightarrow{x_1y_1} &\in P^\delta \mid R\subseteq A(D))\\
  &= \mathbb{P}(\overrightarrow{x_1y_1} \in P^\delta, \alpha\ge1 \mid R\subseteq A(D)) + \mathbb{P}(\overrightarrow{x_1y_1} \in P^\delta, \alpha=0 \mid R\subseteq A(D))\\
  &\leq o(1) + 2^{2i+1} \mathbb{P}(\overrightarrow{x_1y_1} \in Q^\delta) = o(1).
\end{split}
\end{equation}

Now, together with (\ref{conditionalexp}), $\mathbb{E}(F)\leq o(n_1^{i+1}n_2^{i+1}p^{2i+2})$, and by Markov's inequality,
\begin{equation}\label{c3}
\mathbb{P}(F\geq\delta N\Delta)\leq \frac{2^{2i+1}o(n^{2i+2}p^{2i+2})}{\delta(1-\varepsilon_1)n^{2i+2}p^{2i+2}} = o(1).
\end{equation}
This means condition (3) holds for $\mathcal{H}$ a.a.s.

We are now ready to apply Theorem \ref{thm:hypergraphmatching}. The theorem tells us that for sufficiently large $n$, there exists a matching $M$ of $\mathcal{H}$ of size at least $(1-\varepsilon_1)\frac{N}{2i+2}$. Therefore $M^{-1}=\{ H^{-1}\mid H\in M\}$ is a matching on $\mathcal{H}^{-1}$ defined on $D^{-1}$. Again, by Theorem \ref{thm:hypergraphmatching}, we have another matching $M^\prime$ in $\mathcal{H}^{-1}$ of size at least $(1-\varepsilon_1)\frac{N}{2i+2}$ such that $M^\prime\cap M^{-1}=\varnothing$. This implies that $M\cup M^\prime$ does not have non-simple blossoms of length $2$.

Next we will argue that there is only a small number of simple blossoms. Consider the digraph $D\cup D^{-1}$. Let $2\leq j\leq\frac{1}{\varepsilon_1}$ be an integer, and let $T(j)$ be the number of simple $(2i+2)$-blossoms of length $j$ in $D\cup D^{-1}$. We have
\begin{equation}\label{Tj}
\begin{split}
\mathbb{E}(T(j))&\leq n_1(n_1^{ij}n_2^{ij}p^{j+2ij})+n_2(n_1^{ij}n_2^{ij}p^{j+2ij})\\
&\leq 2n_1(n_1^{ij}n_2^{ij}p^{j+2ij})\leq 2\sqrt cn^{1+2ij}p^{j+2ij}\\
&< 2\sqrt{c}\,n^2p n^{\frac{1}{\varepsilon_1}(2i-\varepsilon_1)} p^{\frac{1}{\varepsilon_1}(2i+1-\varepsilon_1)}\\
&<2\sqrt{c}\,n^2 pn^{\frac{1}{\varepsilon_1}(2i-\varepsilon_1)} n^{-\frac{1}{\varepsilon_1}(2i-\varepsilon_1)} = O(n^2p).
\end{split}
\end{equation}
Hence by Markov's inequality,
\begin{equation}
\mathbb{P}\biggl(\,\sum_{j=2}^{1/\varepsilon_1}T(j)\geq\varepsilon_1 pn^2\biggl) \leq \mathbb{P}\bigl(T(j)\geq\varepsilon_1^2pn^2\bigl)\leq o(1).
\end{equation}
Therefore, a.a.s. the number of simple $(2i+2)$-blossoms of length at most $1/\varepsilon_1$ in $D\cup D^{-1}$ is at most $\varepsilon_1 pn_1n_2=\varepsilon_1 pn^2$. Since $M\cup M^{\prime}$ has size at least $2(1-\varepsilon_1)\frac{N}{2i+2}$, it has a subset $M_1$ without simple $(2i+2)$-blossom of length at most $1/\varepsilon_1$ after removing at most $\varepsilon_1pn^2$ closed trails. By using (3) we have:
\begin{equation}
\begin{split}
|M_1| &\geq 2(1-\varepsilon_1)\frac{N}{2i+2}-\varepsilon_1 pn^2\\
&\geq(1-\varepsilon_1)\frac{N}{i+1}-\frac{\varepsilon_1}{1-\varepsilon_1}N\\
&\geq\Big(1-\frac{i+2}{1-\varepsilon_1}\varepsilon_1\Big)\frac{N}{i+1}, \qquad\text{a.a.s}.
\end{split}
\end{equation}
Now we consider the $(2i+2)$-blossoms of length at least $1/\varepsilon_1$ in $M_1$. If $\mathcal{C}_1$ and $\mathcal{C}_2$ are two blossoms of $M_1$ with center $v$, by the way we constructed $M_1$ we could see that the tips of $\mathcal{C}_1$ and $\mathcal{C}_2$ cannot intersect. Therefore, if $v$ has $m$ neighbours in $D$, at most $\varepsilon_1 m$ different $(2i+2)$-blossoms of length at least $1/\varepsilon_1$ have center $v$. Thus, the total number of such blossoms is at most $\sum_{v\in V(D)}\deg_{G}(v)/(1/\varepsilon_1)=2\varepsilon_1N$. By removing one of the trails from each such blossom we get a blossom-free subset $M_0\subseteq M_1$ which satisfies a.a.s.
\begin{equation}\label{blossom-free}
\begin{split}
|M_0|&\geq|M_1|-2\varepsilon_1 N\\
&\geq\Big(1-\frac{3i+4}{1-\varepsilon_1}\varepsilon_1\Big)\frac{N}{i+1}=(1-\varepsilon_0)\frac{N}{i+1}.
\end{split}
\end{equation}
Finally, using $M_0$ we can obtain an $\varepsilon_0$-near $(2i+2)$-gon embedding of $G$ a.a.s. by using Lemma \ref{lem:blossom}. This completes the proof when $p\ll n^{-\frac{2i-\varepsilon_1}{2i+1-\varepsilon_1}}$.

For the case $p\geq \Theta\bigl(n^{-\frac{2i-\varepsilon_1}{2i+1-\varepsilon_1}}\bigr)$, we use a similar argument as used in \cite[Lemma 4.8]{genus}. Choose an integer $t=t(n)$, such that $n^{-\frac{2i}{2i+1}}\ll p/t\ll n^{-\frac{2i-\varepsilon_1}{2i+1-\varepsilon_1}}$. Let $p_1=p/t$. Now take a corresponding digraph $D$ of $\mathcal{G}_{n_1,n_2,p}$ and partition its edges into $t$ parts, putting each edge in one of the parts uniformly at random. Then each of the resulting digraphs $D_1,D_2,\dots, D_t$ is a corresponding digraph of $\mathcal{G}_{n_1,n_2,p_1}$. By the above, for every $1\leq j\leq t$, $D_j\cup D_j^{-1}$ has a collection of blossom-free directed $(2i+2)$-trails of size at least $(1-\varepsilon_0)\frac{|A(D_j)|}{i+1}$ a.a.s. That means, if we let $q$ be the probability that $D_{j}\cup D_{j}^{-1}$ does not have such a set of trails, then $q\to0$ as $n\to\infty$.

Let $I\subseteq\{1,2,\dots,t\}$ be the index set, containing all $j$, $1\leq j\leq t$, for which $D_{j}\cup D_{j}^{-1}$ does not have a collection of directed blossom-free $(2i+2)$-trails of size at least $(1-\varepsilon_0)\frac{|A(D_j)|}{i+1}$. Then by Markov's inequality, $\mathbb{P}(|I|\geq\sqrt qt)\leq \sqrt q$. Hence for sufficiently large $n$, $|I|\leq\varepsilon_0t/2$ a.a.s.

Similarly as in the proof of (\ref{N}), we see that for each $1\leq j\leq t$ a.a.s.
\begin{equation}\label{eq:19}
(1-\varepsilon_0)n^2p_1\leq|A(D_j)|\leq(1+\varepsilon_0)n^2p_1.
\end{equation}
Let $q'$ be the probability that (\ref{eq:19}) does not hold for $|A(D_j)|$. Then $q'\to0$ as $n\to\infty$. Let $I'\subseteq\{1,2,\dots,t\}$ such that for every $j\in I'$, $|A(D_j)|$ does not satisfy (\ref{eq:19}). Thus again by Markov's inequality, $\mathbb{P}(|I'|\geq\sqrt{q'}t)\leq\sqrt{q'}$. Therefore $|I'|\leq\varepsilon_0t/2$ a.a.s.~when $n$ is large enough.

Now let $\Gamma$ be the union of collections of directed blossom-free $(2i+2)$-trails of size at least $(1-\varepsilon_0)\frac{|A(D_j)|}{i+1}$ for $j\notin I,I'$. We have:
\begin{equation}
\begin{split}
|\Gamma|&\geq(1-\varepsilon_0)\sum_{j\notin I,I'}\frac{|A(D_j)|}{i+1}\geq(1-\varepsilon_0)^2t(1-\varepsilon_0)\frac{p_1n^2}{i+1}\\
&\geq(1-\varepsilon_0)^3\frac{pn^2}{i+1}\\
&\geq(1-\varepsilon)(1+\varepsilon_0)\frac{pn^2}{i+1}\geq(1-\varepsilon)\frac{|A(D)|}{i+1}.
\end{split}
\end{equation}
Since the directed closed trails of $\Gamma$ that belong to any $D_j$ $(j\notin I)$ are blossom-free and any $D_k$ and $D_j$ are edge disjoint for $k\neq j$, $\Gamma$ is blossom-free. By Lemma \ref{lem:blossom}, we get a rotation system $\Pi$ in which every closed trail in $\Gamma$ is a face of $\Pi$. Let $f_{2i+2}$ be the number of faces of length $2i+2$ of $\Pi$. We have $(2i+2)f_{2i+2}\geq2(1-\varepsilon)|E(G)|$, thus $\Pi$ is an $\varepsilon$-near $(2i+2)$-gon embedding a.a.s.
\end{proof}

The result of Lemma \ref{lem:biembed} has been proved under the assumption that $n_2=\Theta(n_1)$. However, that assumption can be omitted as long as $n_2\gg 1$.

\begin{lemma}\label{lem:general}%
Let $\varepsilon>0$ and $G\in\mathcal{G}_{n_1,n_2,p}$ be a random bipartite graph on vertex set $X\sqcup Y$ with $|X|=n_1\geq n_2=|Y|$. If $p\gg n_2^{-\frac{2i}{2i+1}}$ where $i$ is a fixed positive integer and $n_2\gg1$, then a.a.s. (as $n_2\to\infty$) $G$ has an $\varepsilon$-near $(2i+2)$-gon embedding.
\end{lemma}

\begin{proof}
It is sufficient to consider the case $n_1/n_2\gg1$. Let $t=\lfloor \frac{n_1}{n_2}\rfloor$, and let $\mathcal{P}=\{X_j\}_{j\in J}$ be the equitable partition of $X$ into $t$ parts, where $J=[t]$. Note that $|X_j|=N_j$ is between $n_2$ and $2n_2$, for every $j\in J$. Let $G_j$ be the bipartite graph $G[X_j\sqcup Y]$ and let $D_j$ be its corresponding digraph. Choose $\varepsilon_0>0$ such that $\varepsilon>\frac{4\varepsilon_0}{1+\varepsilon_0}$. By the proof of Lemma \ref{lem:biembed} there exists a set $M_j$ of closed trails of length $2i+2$ in $D_j\cup D_j^{-1}$, such that $|M_j|\geq(1-\varepsilon_0)\frac{|A(D_j)|}{i+1}$ and $M_j$ is blossom-free, for each $j\in J$ a.a.s. That means, if we let $q_j$ be the probability that $D_j\cup D_j^{-1}$ does not have such a set of closed trails, we have $q_j\to 0$ when $n_2\to \infty$. The probabilities $q_j$ are almost the same since $|X_j|$ only takes at most two different values. We let $q=\max\{q_j\mid j\in J\}$. Define the index set $I\subseteq J$ containing those $j\in J$, for which $D_j\cup D_j^{-1}$ does not have a set of closed trails satisfying the conditions stated above. By Markov's inequality, we have $\mathbb{P}(|I|\geq\sqrt qt)\leq\sqrt q$. Then, when $n_2$ is large enough, $|I|\leq\varepsilon_0t/2$. 

Let $I'$ be the index set such that for every $j\in I'$ either $|A(D_j)|>(1+\varepsilon_0)N_jn_2p$ or $|A(D_j)|<(1-\varepsilon_0)N_jn_2p$. Thus by Markov's inequality we may assume $|I'|<\varepsilon_0t/2$ a.a.s.~when $n_2$ is large enough.

Similarly as in the proof of (\ref{N}) we have a.a.s.
\begin{equation}
\begin{split}
(1-\varepsilon_0)N_jn_2p&\leq |A(D_j)|\leq(1+\varepsilon_0)N_jn_2p,\ \ \forall j\in J\setminus I',\\
(1-\varepsilon_0)n_1n_2p&\leq |E(G)|\leq(1+\varepsilon_0)n_1n_2p.
\end{split}
\end{equation}
Let $M=\bigcup_{j\in J\setminus (I\cup I')}M_j$. Since each $M_j$ $(j\in J\setminus (I\cup I'))$ is blossom-free and the edge-sets of different $D_j$ are disjoint, $M$ is also blossom-free. We also have:
\begin{equation}
\begin{split}
|M|&=\sum_{j\in J\setminus (I\cup I')}|M_j|\geq t(1-\varepsilon_0)(1-\varepsilon_0)\frac{|A(D_j)|}{i+1}\\
&\geq t(1-\varepsilon_0)^3\frac{N_jn_2p}{i+1}\geq(1-\varepsilon_0)^4\frac{n_1n_2p}{i+1}\\
&\geq(1+\varepsilon_0)(1-\varepsilon)\frac{n_1n_2p}{i+1}\geq(1-\varepsilon)\frac{|E(G)|}{i+1}.
\end{split}
\end{equation}
Therefore, by Lemma \ref{lem:blossom} we get the desired $\varepsilon$-near $(2i+2)$-gon embedding $\Pi$ a.a.s.
\end{proof}

We are ready to complete the proof of our first main result. For reader's reference, Theorem~\ref{thm:1.3} is restated here.
\begin{thm:associativity}
Let $\varepsilon>0$ and $G\in\mathcal{G}_{n_1,n_2,p}$ be a random bipartite graph and suppose that $i\geq2$ is an integer. If $p$ satisfies $(n_1n_2)^{-\frac{i}{2i+1}}\ll p\ll (n_1n_2)^{-\frac{i-1}{2i-1}}$, $n_1/n_2<c$ and $n_2/n_1<c$ where $c$ is a positive real number, then we have a.a.s.
\begin{equation*}
(1-\varepsilon)\frac{i}{2i+2}pn_1n_2\leq g(G)\leq(1+\varepsilon)\frac{i}{2i+2}pn_1n_2
\end{equation*}
and
\begin{equation*}
(1-\varepsilon)\frac{i}{i+1}pn_1n_2\leq \widetilde{g}(G)\leq(1+\varepsilon)\frac{i}{i+1}pn_1n_2.
\end{equation*}
\end{thm:associativity}
\begin{proof}
To prove the lower bound, we count the number of closed trails of $G$ of length at most $2i$. Let $C$ be the number of such closed trails. We have
\begin{equation}
\mathbb{E}(C)\leq\sum_{j=2}^{i}n_1^jn_2^jp^{2j}=o(n_1n_2p).
\end{equation}
Then by Markov's inequality, a.a.s. at most $\frac{1}{4(i-1)}\varepsilon pn_1n_2$ closed trails of $G$ have length at most $2i$. Similarly as in the proof of (\ref{N}) we get $|E(G)|\geq(1-\frac{1}{2i}\varepsilon)pn_1n_2$, a.a.s. Let $\Pi$ be a rotation system of $G$, and let $f(\Pi)$ be the number of faces, and $f^\prime$ be the number of faces of $\Pi$ with length at most $2i$. Then $2|E(G)|\geq(2i+2)(f(\Pi)-f^\prime)+4f^\prime\geq(2i+2)f(\Pi)-(2i-2)f^\prime$. By the above, $f^\prime\leq 2C\leq\frac{1}{2(2i-2)}\varepsilon pn_1n_2$ a.a.s. Now we have a.a.s.
\begin{equation}\label{24}
\begin{split}
g(G,\Pi)&=\frac{1}{2}(|E(G)|-f(\Pi)-|V(G)|)+1\sim\frac{1}{2}(|E(G)|-f(\Pi))\\
&\geq\frac{i}{2i+2}|E(G)|-\frac{i-1}{2i+2}f^\prime\\
&\geq\Bigl(1-\frac{1}{2i}\varepsilon\Bigr) \frac{i}{2i+2}pn_1n_2-\frac{i-1}{2i+2}\frac{1}{4(i-1)}\varepsilon pn_1n_2\\
&\geq(1-\varepsilon)\frac{i}{2i+2}pn_1n_2.
\end{split}
\end{equation}
For the upper bound, by Lemma \ref{lem:biembed} we have an $\varepsilon^\prime$-near $(2i+2)$-gon embedding $\Pi$ a.a.s., with $\varepsilon^\prime=\frac{i\varepsilon}{2+\varepsilon}$, and let $f(\Pi)$ be the number of faces. Also, we have a.a.s. $|E(G)|\leq(1+\frac{1}{2}\varepsilon)pn_1n_2$. Therefore, a.a.s.
\begin{equation}\label{25}
\begin{split}
g(G,\Pi)&=\frac{1}{2}(|E(G)|-f(\Pi)-|V(G)|)+1\sim\frac{1}{2}(|E(G)|-f(\Pi))\\
&\leq\frac{1}{2}(|E(G)|-\frac{2(1-\varepsilon^\prime)}{2i+2}|E(G)|)\\
&\leq\Bigl(1+\frac{1}{2}\varepsilon\Bigr)\frac{i+\varepsilon^\prime}{2i+2}pn_1n_2 = (1+\varepsilon)\frac{i}{2i+2}pn_1n_2.
\end{split}
\end{equation}
This completes the proof for the orientable genus. The proof for $\widetilde{g}(G)$ is essentially the same, where the lower bound uses Euler's Formula as in (\ref{24}), while for the upper bound we just observe that $\widetilde{g}(G)\leq2g(G)+1$, see \cite{top}.
\end{proof}

\begin{thm:associativity2}
Let $\varepsilon>0$ and $G\in\mathcal{G}_{n_1,n_2,p}$ be a random bipartite graph. If $n_1\geq n_2\gg1$ and $p\gg n_2^{-\frac{2}{3}}$, then we have a.a.s.
\begin{equation*}
(1-\varepsilon)\frac{pn_1n_2}{4}\leq g(G)\leq(1+\varepsilon)\frac{pn_1n_2}{4}
\end{equation*}
and
\begin{equation*}
(1-\varepsilon)\frac{pn_1n_2}{2}\leq \widetilde{g}(G)\leq(1+\varepsilon)\frac{pn_1n_2}{2}.
\end{equation*}
\end{thm:associativity2}
\begin{proof}
The lower bound follows from \cite[Proposition 4.4.4]{top}. For the upper bound, we have the same proof as for (\ref{25}), except that we use Lemma \ref{lem:general} (with $i=1$) instead of Lemma \ref{lem:biembed}.
\end{proof}

\section{Random bipartite graphs with a small part}

Now we consider the case when $G\in\mathcal{G}_{n_1,n_2,p}$ where $n_1\gg 1$ and $n_2$ is a constant. We say $S$ is a {\em standard graph} of $\mathcal{G}_{n_1,n_2,p}$ if $S$ is a bipartite graph on the vertex set $V(S)=X\sqcup Y$ with $|X|\sim n_1$ and $|Y|=n_2$, and we have expected degree distributions for $\mathcal{G}_{n_1,n_2,p}$. This means, for every $Y^\prime\subseteq Y$ with $|Y^\prime|=m$, $|\{x\in X\mid N(x)=Y^\prime\}|=\lfloor p^m(1-p)^{n_2-m}n_1\rfloor$, where $N(x)$ is the set of neighbours of $x$. Note that the standard graph is essentially unique. Suppose that $c$ is some constant. Then we say that an embedding $\Pi$ of $G$ is a {\em near $k$-gon embedding} (\emph{with respect to} $c$) if $2|E(G)|-kf_k(\Pi)\leq c$.

\begin{lemma}\label{lem:standard}
Let $S$ be the standard graph of $\mathcal{G}_{n_1,n_2,p}$ where $n_1\gg 1$ and $n_2$ is a constant. Suppose that $p\gg n_1^{-\frac{1}{2}}$ and let $S^\prime$ be the bipartite graph obtained by removing all vertices of degree at most one in $S$. Then $S^\prime$ has a near $4$-gon embedding with respect to the constant $c=(4n_2+14)2^{n_2}$.
\end{lemma}

\begin{proof}
Let $V(S)=X(S)\sqcup Y(S)$. Note that $n_1-2^{n_2}\leq |X(S)|\leq n_1$ and $|Y(S)|=n_2$. For every $Y^\prime\subseteq Y(S)$, let $F_S(Y^\prime)=\{x\in X(S)\mid N(x)=Y^\prime\}$. Now consider all of the $2^{n_2}$ subsets of $Y(S)$, they give us a partition of $X(S)=\bigsqcup_{Y^\prime\subseteq Y(S)} F_S(Y^\prime)$. Note that $S[Y^\prime\sqcup F_S(Y^\prime)]$ is a complete bipartite graph for every $Y^\prime\subseteq Y(S)$. If $|Y^\prime|\geq2$, by \cite{Ringel}, we have a near $4$-gon embedding of $S[Y^\prime\sqcup F_S(Y^\prime)]$. Moreover, there is always a near $4$-gon embedding with respect to the constant $14$ since in the worst case, we may have one $6$-gon and one $8$-gon apart from the $4$-gons. Let $\mathcal{C}(Y^\prime)$ be the set of all facial walks of length $4$ in the optimal embedding of $S[Y^\prime\sqcup F_S(Y^\prime)]$. We can remove from $\mathcal{C}(Y^\prime)$ a collection of at most $|Y^\prime|$ closed trails to make $\mathcal{C}(Y^\prime)$ free of blossoms with center in $Y^\prime$. Therefore, we can remove at most $2^{n_2}n_2$ closed trails of length $4$ to make $\bigcup_{Y^\prime\in Y(S), |Y^\prime|\geq2}\mathcal{C}(Y^\prime)$ free of blossoms centered in $Y$. Now we apply Lemma \ref{lem:blossom} for the union of these sets for all $Y^\prime$ with $|Y^\prime|\geq2$. This shows that there is a near $4$-gon embedding of $S^\prime$ with respect to the constant $c=(4n_2+14)2^{n_2}$.
\end{proof}

\begin{lemma}\label{lem:g(S)}
Let $S$ be the standard graph of $\mathcal{G}_{n_1,n_2,p}$ where $n_1\gg 1$ and $n_2$ is a constant. Suppose that $p\gg n_1^{-\frac{1}{3}}$, then
\begin{equation*}
g(S)\sim\frac{n_1n_2p}{4}\sum_{i=2}^{n_2-1}\frac{i-1}{i+1} \binom{n_2-1}{i}(-p)^i.
\end{equation*}
In particular, when $n_1^{-\frac{1}{3}}\ll p\ll1$, $g(S)=(1+o(1))\frac{n_1p^3}{4}\binom{n_2}{3}$.
\end{lemma}

\begin{proof}
Let $\Pi$ be the rotation system of $S^\prime$ given by Lemma \ref{lem:standard}. Since this gives a near 4-gon embedding, we have the following (Note that $|X(S')|\sim n_1$ by the definition of the standard graph, we will  use $n_1$ in the calculation):
\begin{equation}
\begin{split}
g(S^\prime)&\sim\frac{1}{2}(2+|E(S^\prime)|-f(\Pi)-|V(S^\prime)|)\\
&\sim\frac{1}{2}(|E(S^\prime)|-f(\Pi)-n_1+(1-p)^{n_2}n_1+(1-p)^{n_2-1}n_1n_2p)\\
&\sim\frac{1}{2}\biggl(\frac{|E(S^\prime)|}{2}-n_1+(1-p)^{n_2}n_1+(1-p)^{n_2-1}n_1n_2p\biggr)\\
&\sim\frac{1}{2}\biggl(\frac{n_1n_2p-(1-p)^{n_2-1}n_1n_2p}{2}-n_1+(1-p)^{n_2}n_1+(1-p)^{n_2-1}n_1n_2p\biggr)\\
&=\frac{1}{2}\biggl(\frac{1}{2}n_1n_2p\sum_{i=1}^{n_2-1}\binom{n_2-1}{i}(-p)^i+n_1\sum_{i=2}^{n_2}\binom{n_2}{i}(-p)^i\biggr)\\
&=\frac{n_1n_2p}{4}\sum_{i=2}^{n_2-1}\binom{n_2-1}{i}(-p)^i\frac{i-1}{i+1}.
\end{split}
\end{equation}
Since $g(S)=g(S^\prime)$, this completes the proof.
\end{proof}

\begin{lemma}\label{lem:4.3}
Let $\varepsilon>0$ and $G\in\mathcal{G}_{n_1,n_2,p}$ where $n_1\gg 1$ and $n_2$ is a constant. If $p\gg n_1^{-\frac{1}{3}}$ and $S$ is the standard graph of $\mathcal{G}_{n_1,n_2,p}$, we have a.a.s. (as $n_1\to\infty$)
\begin{equation*}
(1-\varepsilon)g(S)\leq g(G)\leq(1+\varepsilon)g(S).
\end{equation*}
\end{lemma}

\begin{proof}
Let $V(G)=X(G)\sqcup Y(G)$ with $|X(G)|=n_1$ and $|Y(G)|=n_2$. For every $Y'\subseteq Y(G)$, where $|Y'|=m\geq1$, let $F_G(Y')=\{x\in X(G)\mid N(x)=Y'\}$. Then
\begin{equation}
\begin{split}
\mathbb{E}(|F_G(Y')|)&=p^m(1-p)^{n_2-m}n_1,\\
\mathbb{E}(|F_G(Y')|^2)&=p^{2m}(1-p)^{2n_2-2m}n_1(n_1-1)+p^m(1-p)^{n_2-m}n_1.
\end{split}
\end{equation}
For every $t>0$, by Chebyshev's inequality, we have
\begin{equation}\label{differ}
\begin{split}
\mathbb{P}\bigl(\bigl||F_G(Y')|&-\mathbb{E}(|F_G(Y')|)\bigr|\geq t\mathbb{E}(|F_G(Y')|)\bigr) \leq \frac{\mathbb{E}(|F_G(Y')|^2) - \mathbb{E}^2(|F_G(Y')|)}{t^2\mathbb{E}^2(|F_G(Y')|)} \\[1.5mm]
&\sim \frac{p^m(1-p)^{n_2-m}n_1}{t^2p^{2m}(1-p)^{2n_2-2m}n_1^2}
= \frac{1}{t^2p^m(1-p)^{n_2-m}n_1}\, .
\end{split}
\end{equation}

Suppose now that $p\gg n_1^{-\frac{1}{3}}$ and $m\geq3$. By taking $t = \tfrac{\varepsilon}{10n_2 2^{n_2}} p^{3-m}(1-p)^{(m-n_2)/2}$ in (\ref{differ}) we obtain that
\begin{equation}\label{differ2}
\begin{split}
\mathbb{P}\bigl(\bigl||F_G(Y')|-\mathbb{E}(|F_G(Y')|)\bigr| \geq \tfrac{\varepsilon}{10n_2 2^{n_2}} p^3(1-p)^{(n_2-m)/2}n_1\bigr)
\leq \frac{100n_2^2 4^{n_2}}{\varepsilon^2 p^{6-m}n_1}\leq o(1).
\end{split}
\end{equation}

Let $S$ be the standard graph of $\mathcal{G}_{n_1,n_2,p}$ with $V(S)=X(S)\sqcup Y(S)$. We may assume that $Y(S)=Y(G)=[n_2]$. Let $G'$ be the subgraph obtained from $G$ by deleting all vertices of degree at most 2 in $X(G)$. Observe that deleting vertices of degree at most 1 does not change the genus and that vertices of degree 2 form (at most) $\binom{n_2}{2}$ ``parallel'' classes, thus
\begin{equation}\label{eq:g and G'}
 g(G')\le g(G)\le g(G') + \binom{n_2}{2}.
\end{equation}

For every $Y'\subseteq Y$ (and $|Y'|\ge3$), we consider $F_{G'}(Y') = F_G(Y')$ and $F_S(Y')$. By (\ref{differ2}), these two sets have almost the same cardinality (a.a.s.). More precisely, we have a.a.s.
\begin{equation}\label{eq:G' vs S}
\begin{split}
  \sum_{|Y'|\ge3} \bigl| |F_G(Y')| - |F_S(Y')| \bigr|
  &\le \sum_{|Y'|\ge3} \bigl(\bigl| |F_G(Y')|-\mathbb{E}(|F_G(Y')|)\bigr|+1\bigr)\\[1mm]
  &\le 2^{n_2} (1 + \frac{\varepsilon}{2^{n_2}10n_2}p^3(1-p)^{(n_2-m)/2}n_1) \\
  &\le 2^{n_2} + \frac{\varepsilon}{10n_2}\,p^3 n_1.
\end{split}
\end{equation}

We first consider the case $p\ll 1$. Let $S^*$ be the subgraph obtained from $S$ by deleting all vertices of degree at most $2$ in $X(S)$. By a similar argument we used in (\ref{eq:g and G'}), we have $|g(S)-g(S^*)|=O(1)$. Then (\ref{eq:G' vs S}) implies, in particular, that $S^*$ can be obtained from $G'$ by adding and deleting at most $n_2 2^{n_2} + \tfrac{\varepsilon}{10}p^3 n_1$ edges a.a.s. Since adding an edge changes the genus by at most 1, and by Lemma \ref{lem:g(S)}, $g(S)> \tfrac{1}{5}p^3n_1\gg 1$ (if $n_1$ is large), we obtain that
$(1-\tfrac{1}{2}\varepsilon)g(S)\leq g(G')\leq(1+\tfrac{1}{2}\varepsilon)g(S)$ a.a.s.
Together with (\ref{eq:g and G'}) this implies the lemma.

Now suppose that $p=\Theta(1)$. Assume first that $(1-p)^{n_2-3}n_1\gg1$. In this case we take $t = \tfrac{\varepsilon p \Psi(p,n_2)}{15n_2 2^{n_2}}$ in (\ref{differ}), where $\Psi(p,n_2)$ is defined in Theorem \ref{thm:1.5}. Note that $t=O(1)$. Therefrom we conclude that with high probability
\begin{equation*}
\bigl| |F_G(Y')| - |F_S(Y')| \bigr| \leq
\frac{\varepsilon p \Psi(p,n_2)}{15n_2 2^{n_2}} |F_S(Y')|+O(1).
\end{equation*}
Now we derive similarly as above that $S$ can be obtained from $G$ by adding and removing less than $2n_2 2^{n_2} + \tfrac{\varepsilon}{10}\,p \Psi(p,n_2)n_1$ edges a.a.s., which is less than $\tfrac{\varepsilon}{5}\,p \Psi(p,n_2)n_1$ when $n_1$ is sufficiently large. The same conclusion as above follows.

Finally, suppose that $(1-p)^{n_2-3}n_1=O(1)$. By Lemma \ref{lem:g(S)} we have $g(S')\sim\frac{n_1}{4}(n_2-2).$ Thus by using a similar argument as we used above, (\ref{eq:G' vs S}) implies that $S^*$ can be obtained from $G'$ by adding or deleting at most $\varepsilon n/2$ edges a.a.s. Thus the same conclusion follows.
\end{proof}

We have all tools to prove the last main statement.

\begin{thm:associativity3}
Let $\varepsilon>0$ and $G\in\mathcal{G}_{n_1,n_2,p}$ where $n_1\gg 1$ and $n_2\geq3$ is a constant.
\begin{itemize}
\item[\rm{(a)}]
If $p\gg n_1^{-\frac{1}{3}}$ we have a.a.s. (as $n_1\to\infty$)
\begin{equation*}
(1-\varepsilon)\frac{n_1n_2p}{4}\,\Psi(p,n_2)\leq g(G)\leq(1+\varepsilon)\frac{n_1n_2p}{4}\,\Psi(p,n_2)
\end{equation*}
and
\begin{equation*}
(1-\varepsilon)\frac{n_1n_2p}{2}\,\Psi(p,n_2)\leq \widetilde{g}(G)\leq(1+\varepsilon)\frac{n_1n_2p}{2}\,\Psi(p,n_2),\\[1mm]
\end{equation*}
where $\Psi(p,n_2)=\sum_{i=2}^{n_2-1}\frac{i-1}{i+1}\binom{n_2-1}{i}(-p)^i$.
\item[\rm{(b)}]
If $n_1^{-\frac{1}{2}}\ll p\ll n_1^{-\frac{1}{3}}$, then a.a.s.
\begin{equation*}
g(G)=\biggl\lceil\frac{(n_2-3)(n_2-4)}{12}\biggr\rceil\quad and\quad \widetilde{g}(G)=\biggl\lceil\frac{(n_2-3)(n_2-4)}{6}\biggr\rceil
\end{equation*}
with a single exception that $\widetilde{g}(G)=3$ when $n_2=7$.
\item[\rm{(c)}]
If $p\ll n_1^{-\frac{1}{2}}$, then a.a.s. $g(G)=0$.
\end{itemize}
\end{thm:associativity3}
\begin{proof}
To prove part (a), we just combine Lemmas \ref{lem:g(S)} and \ref{lem:4.3}. The results for the non-orientable genus can be obtained by adding a crosscap to the orientable surface in which $G$ has a minimum genus embedding.

For case (b), when $Y'\subseteq Y(G)$ with $|Y'|=m\geq3$, we have
\begin{equation}
\mathbb{E}(|F_G(Y')|)=p^m(1-p)^{n_2-m}n_1=o(1).
\end{equation}
Then by Markov's inequality, $\mathbb{P}(|F_G(Y')|\geq1)=o(1)$. For the sets $Y_2\subseteq Y(G)$ with $|Y_2|=2$, by (\ref{differ}) we can see that for every $t>0$, $(1-t)p^2n_1\leq |F_G(Y_2)|\leq (1+t)p^2n_1$ a.a.s. That means if we remove all vertices with degree $1$ in $G$, we will obtain the complete graph $K_{n_2}$, in which each edge is replaced by roughly $p^2n_1$ internally disjoint paths of length $2$. By \cite{complete} we have $g(G)=g(K_{n_2})=\Bigl\lceil\frac{(n_2-3)(n_2-4)}{12}\Bigr\rceil$ a.a.s. (and similarly for $\tilde g(G)$, where the exception occurs when $n_2=7$).
This proves part (b).

To prove (c), note that when $p\ll n_1^{-1/2}$, none of the subdivided edges of $K_{n_2}$ from case (b) will occur (a.a.s.), and with high probability, every vertex in $X(G)$ will be of degree at most 1. Thus, $g(G)=0$ a.a.s.
\end{proof}

Note that in Theorem \ref{thm:1.5}, when $p=\Theta(n^{-\frac{1}{3}})$, a.a.s. the graph $G$ will be the Levi graph of $mK_{n_1}^3$, where $mK_{n_1}^3$ is the complete $3$-uniform multi-hypergraph of order $n_1$, and each triple has $m$ edges. This problem is hard and of independent interest as a generalization of Ringel-Youngs Theorem. We will discuss it in a separate paper \cite{JM2}.

\bibliographystyle{abbrv}
\bibliography{reference}

\end{document}